\documentclass[12pt]{amsart}

\usepackage[margin=1.5in]{geometry}
\usepackage[T1]{fontenc}
\usepackage[utf8]{inputenc}
\usepackage[pdftex]{graphicx}
\usepackage{tikz-cd}
\usepackage{blindtext}

\usepackage{hyperref,xcolor}
\usepackage{amsmath,amsfonts,amssymb,amsthm,color,mathtools,enumitem}
\usepackage{epstopdf}
\usepackage{polynom}
\usepackage[english]{babel}
\usepackage{mathrsfs}
\usepackage{chngcntr}
\usepackage{verbatim} 

\hypersetup{
	colorlinks=true,
	linkcolor=blue,
	citecolor=cyan,
	urlcolor=cyan}

\newcommand{\ov}{\overline}
\newcommand{\bN}{\mathbb{N}}
\newcommand{\bR}{\mathbb{R}}
\newcommand{\bZ}{\mathbb{Z}}

\newcommand{\bC}{\mathbb{C}}

\newcommand{\bT}{\mathbb{T}}

\newcommand{\cZ}{\mathcal{Z}}

\newcommand{\ee}{\varepsilon}

\newcommand{\dlim}{\underset{\to}{\lim}}

\newcommand{\act}{\curvearrowright}

\newcommand{\Tr}{\text{Tr}}

\newcommand{\ku}{K_1^{alg,u}}
\newcommand{\ka}{K_1^{\text{alg}}}
\newcommand{\hku}{\ov{K}_1^{alg,u}}
\newcommand{\hka}{\ov{K}_1^{\text{alg}}}

\newcommand{\tD}{\tilde{\Delta}}

\newcommand{\tr}{\text{tr}}

\newcommand{\id}{\text{id}}

\DeclareMathOperator{\Aff}{Aff}

\newtheorem{theorem}{Theorem}[section]

\newtheorem{prop}[theorem]{Proposition}
\newtheorem{cor}[theorem]{Corollary}
\newtheorem{lemma}[theorem]{Lemma}
\newtheorem{example}[theorem]{Example}
\newtheorem{remark}[theorem]{Remark}

\numberwithin{equation}{section}

\begin{document}

\renewcommand*{\thetheorem}{\Alph{theorem}}

\title{Polar decomposition in algebraic $K$-theory}
\author[Pawel Sarkowicz]{Pawel Sarkowicz}
\email[Pawel Sarkowicz]{\href{mailto:psark007@uottawa.ca}{psark007@uottawa.ca}}
%\address[Pawel Sarkowicz]{Department of Mathematics and Statistics, University of Ottawa, 75 Laurier Ave. East, Ottawa, ON, K1N 6N5 Canada}

\author[Aaron Tikuisis]{Aaron Tikuisis}
\email[Aaron Tikuisis]{\href{mailto:aaron.tikuisis@uottawa.ca}{aaron.tikuisis@uottawa.ca}}
%\address[Aaron Tikuisis]{Department of Mathematics and Statistics, University of Ottawa, 75 Laurier Ave. East, Ottawa, ON, K1N 6N5 Canada}
\address{Department of Mathematics and Statistics, University of Ottawa, 75 Laurier Ave. East, Ottawa, ON, K1N 6N5 Canada}

\begin{abstract}
We show that the Hausdorffized algebraic $K$-theory of a C*-algebra decomposes naturally as a direct sum of the Hausdorffized unitary algebraic $K$-theory and the space of continuous affine functions on the trace simplex.
Under mild regularity hypotheses, an analogous natural direct sum decomposition holds for the ordinary (non-Hausdorffized) algebraic $K$-theory.
\end{abstract}

\maketitle

\tableofcontents{}

\section{Introduction}

$K$-theory is at the heart of understanding the structure of C*-algebras and their morphisms.
Indeed, the recently-established classification of simple separable nuclear C*-algebras shows that an invariant consisting of $K$-theory paired with traces fully determine C*-algebras up to isomorphism, among a large class of ``classifiable'' C*-algebras.
Moreover, a richer $K$-theoretic invariant, augmenting the aforementioned invariant by $K$-theory with coefficients and Hausdorffized unitary $K$-theory, classifies the morphisms between these classifiable C*-algebras, up to approximate unitary equivalence (\cite{CGSTW}).

The components of the richer $K$-theoretic invariant can be directly related, through split exact sequences, to the smaller invariant of $K$-theory and traces.
$K$-theory with $\mathbb Z/n$ coefficients relates to ordinary $K$-theory through the split exact sequence
\begin{equation} 0 \to K_i(A) \otimes \mathbb Z/n \to K_i(A; \mathbb Z/n) \to \mathrm{Tor}(K_{1-i}(A),\mathbb Z/n) \to 0, \end{equation}
where the connecting maps are induced by the Bockstein maps (see \cite[Propositions 1.8 and 2.4]{SchochetIV}).
While the maps in this exact sequence are natural, the splitting is not and cannot be -- as can be seen by the existence of homomorphisms which agree on $K$-theory but not on $K$-theory with coefficients.\footnote{An example is the tensor flip map $\sigma$ on $A:=\mathcal O_3\otimes \mathcal O_3$. Using the K\"unneth formula, one can compute that $K_0(A)=K_1(A)=\bZ/2$, and since the flip map is an order 2 automorphism, it must induce the identity on $K$-theory.
Suppose, for a contradiction, that $\sigma$ induces the identity on $K$-theory with coefficients.
One can see that $\text{Pext}(K_0(A),K_0(A))=0$ and so, since the algebra $A$ satisfies the UCT, it would follow that $[\sigma]=[\mathrm{id}_{A}]$ in $KK(A,A)$ by \cite{DadarlatLoring96} (see Proposition 2.4 in particular).
The Kirchberg--Phillips classification theorem (Theorem 8.3.3 (iii) of \cite{RordamBook}, for example) would then imply that $\sigma$ is approximately inner.
Hence, $\mathcal O_3$ would have approximately inner flip, but it doesn't by \cite{Tikuisis16,EndersSchemaitatTikuisis23}}

Similarly, Hausdorffized unitary $K$-theory, denoted $\ov{K}^{\mathrm{alg},u}_1(A)$, relates to ordinary $K$-theory and traces, through the split exact sequence
\begin{equation}
\label{eq:ThomsenSplitting} 0 \to \Aff T(A)/\overline{\rho(K_0(A))} \to \ov{K}^{\mathrm{alg},u}_1(A) \to K_1(A) \to 0,
\end{equation}
with the first map due to Thomsen (and intimately related to the de la Harpe--Skandalis determinant, as discussed below), and the second map arising as a natural quotient (see \cite[Theorem 3.2 and Corollary 3.3]{Thomsen95}).
Here, $\rho$ is the pairing map, so that $\rho(K_0(A))$ is the subgroup of $\Aff T(A)$ generated by the images of projections.
Again, these connecting maps are natural, though the splitting is not (see \cite[Section 5]{NielsenThomsen96} for example).

A natural question to ask is: why has only \emph{unitary} (Hausdorffized) algebraic $K$-theory been used in the classification theorem and (for the most part) related literature, as opposed to algebraic $K$-theories defined using all invertibles?
A superficial answer may be that, since C*-algebraists are used to working only with unitaries when defining $K_1$, it is natural to work with them in algebraic $K$-theory.
But this answer misses the point, and doesn't tell us whether the classification of morphisms theorem would look any different if all invertibles were used instead of unitaries, in the Hausdorffized algebraic $K$-theory component of the invariant.

In this article, we investigate the relation between (Hausdorffized and ordinary) algebraic $K$-theory defined using unitaries and using general invertible elements.
For topological $K$-theory, the relationship is entirely straightforward: polar decomposition expresses a general invertible as a unitary multiplied by a positive invertible element.
By effectively forgetting the positive invertible part, the unitary group is a strong deformation retract of the invertible group, and therefore the topological $K$-theory is the same whether one uses unitaries or invertibles.
To study algebraic $K$-theory, polar decomposition is once again the crucial tool.
However, since positive invertibles are non-trivial in (even Hausdorffized) algebraic $K$-theory, they need to be carefully accounted for.

Using the polar decomposition, we obtain a decomposition of Hausdorffized algebraic $K$-theory, as our first main result:

\begin{theorem}\label{thm:M1}
Let $A$ be any unital C*-algebra. Then there is a natural isomorphism of topological groups
\begin{equation} \ov{K}_1^{\mathrm{alg}}(A) \simeq \ov{K}_1^{\mathrm{alg,u}}(A) \oplus \Aff T(A). \end{equation}
\end{theorem}

Said differently, the Hausdorffized algebraic $K_1$-class of a positive invertible is determined entirely by tracial data, and the interaction between unitaries and positive invertibles is, at the level of Hausdorffized algebraic $K_1$, entirely trivial.
%The de la Harpe--Skandalis determinant, introduced in \cite{dlHS1}, plays a key role in our analysis (via the exact sequence \eqref{eq:ThomsenSplitting} and an analogous one using $\ov{K}_1^{\mathrm{alg}}$ in place of $\ov{K}_1^{\mathrm{alg,u}}$).
The de la Harpe--Skandalis determinant $\Delta$ (introduced in \cite{dlHS84a}), mapping from the connected component of the unitary group to $\Aff T(A)/\rho(K_0(A))$, provides the inverse of the first map in \eqref{eq:ThomsenSplitting} and plays a key role in our analysis.
This works because the kernel of this determinant (when viewed as a map into $\Aff T(A)/\overline{\rho(K_0(A))}$) consists \emph{exactly} of the closure of the commutator subgroup.

In the non-Hausdorffized setting, the kernel of the de la Harpe--Skandalis determinant (now going into $\Aff T(A)/\rho(K_0(A))$) could be strictly larger than the commutator subgroup (this time, without the closure).
Understanding when these two coincide is a well-studied problem (indeed, two problems, corresponding to the two components identified above for $\ov{K}^{\mathrm{alg}}(A)$), and it has been shown that they coincide under appropriate regularity hypotheses, and in particular for all classifiable C*-algebras.
We state our results for the non-Hausdorffized setting in a general form, without such regularity hypotheses; to do so, we need to work with the quotient of the invertibles (respectively unitaries) by the kernel of the determinant map $\Delta$.

\begin{theorem}\label{thm:M2}
  Let $A$ be any unital C*-algeba. Then there is a natural isomorphism
  \begin{equation} GL_{\infty}(A)/\mathrm{ker}\Delta \simeq U_{\infty}(A)/\mathrm{ker}\Delta|_{U^0_{\infty}(A)} \oplus \Aff T(A). \end{equation}
\end{theorem}

\begin{cor}
Let $A$ be a unital, separable, simple C*-algebra which has stable rank one, is pure in the sense of \cite[Definition 3.6]{Winter12}, and such that every 2-quasitracial state on $A$ is a trace (in particular, $A$ may be any unital, separable, simple, finite, exact, $\mathcal Z$-stable C*-algebra).
Then there is a natural isomorphism
\begin{equation} K_1^{\mathrm{alg}}(A) \simeq K_1^{\mathrm{alg,u}}(A) \oplus \Aff T(A). \end{equation}
\end{cor}

The paper is structured as follows. First we discuss preliminaries and notation in Section \ref{prelim}. We introduce each of the variants of the de la Harpe--Skandalis determinant and discuss some relationships between kernels. In Section \ref{section:polar-decomposition}, we prove the main results (Theorems \ref{thm:M1} and \ref{thm:M2}). In Section \ref{section:pdak-nonstable-algebraic-k-theory} we look at non-stable analogues of the results in \ref{section:polar-decomposition}, under the hypothesis of certain $K$-theoretic regularity conditions. 

It was pointed out that variations of some of our main results were proved as \cite[Theorem 5]{Elliott22}. We use different techniques and exposition.

\addtocontents{toc}{\protect\setcounter{tocdepth}{0}}
\section*{Acknowledgements}
Thanks to George Elliott, Thierry Giordano, Chris Shafhauser, and Stuart White for helpful discussions.
Thanks to the other authors of \cite{CGSTW} for agreeing to share the draft of that article with all authors.

\addtocontents{toc}{\protect\setcounter{tocdepth}{1}}
\section{Preliminaries and notation}\label{prelim}

\renewcommand*{\thetheorem}{\arabic{section}.\arabic{theorem}}

\subsection{Notation} 
%We will use capital letters $A,B,C,D,\dots$ to denote C*-algebras. Generally small letters $a,b,c,d,\dots,x,y,z$ will denote operators in C*-algebras.
For a group $G$, we denote by $DG$ the derived subgroup of $G$, i.e.,
\begin{equation} DG := \langle ghg^{-1}h^{-1} \mid g,h \in G \rangle. \end{equation}
If $G$ has an underlying topology, we denote by $CG$ the closure of $DG$ and $G^0$ the connected component of the identity. 

For a unital C*-algebra $A$, $GL(A)$ denotes the general linear group of $A$, while $GL^0(A)$ denotes the connected component of $GL(A)$. For $n \in \bN$, we write $GL_n(A) := GL(M_n(A))$, $GL_n^0(A) := GL^0(M_n(A))$, and we set
\begin{equation} GL_{\infty}(A) := \dlim\, GL_n(A), \end{equation}
with connecting maps $GL_n(A) \ni x \mapsto x \oplus 1 \in GL_{n+1}(A)$. 
This makes $GL_\infty(A)$ both a topological space (with the inductive limit topology) and a group.\footnote{It is not, however, a topological group. This is addressed in a footnote in \cite{CGSTW}.}
We have unitary analogues by replacing $GL$ with $U$, where $U(A)$ denotes the group of unitaries of $A$. 
Similarly, we define $M_{\infty}(A) = \dlim \, M_n(A)$ (as an algebraic direct limit) with connecting maps $x \mapsto x \oplus 0$. If $E$ is Banach space and $\tau: A \to E$ is a linear map that is tracial (i.e., $\tau(ab) = \tau(ba)$ for all $a,b \in A$), we extend this canonically to $M_{\infty}(A)$ by setting $\tau((a_{ij})) := \sum_i \tau(a_{ii})$ for $(a_{ij})\in M_n(A)$. 

We write $\pi_1(X)$ for the fundamental group of a topological space $X$ with distinguished base point.
(In this paper, we will only use $X$ equal to $U^0_n(A)$ or $GL^0_n(A)$, and in either case, we use the unit element as the base point.)

For a C*-algebra $A$, we let $K_0(A),K_1(A)$ be the topological $K$-groups of $A$. 
%If $A$ is unital, $K_0(A)$ is the Grothendieck group of the Murray-von Neumann semigroup of projections in $M_{\infty}(A)$ and
%\begin{equation} K_1(A) = U_{\infty}/U_{\infty}^0(A) \simeq GL_{\infty}/GL_{\infty}^0(A). \end{equation}
The set of tracial states on $A$ will be denoted $T(A)$, which is a Choquet simplex (\cite[Theorem 3.1.18]{Sakaibook}).
We denote by $\Aff T(A)$ the set of continuous affine functions $T(A) \to \bR$, which is an interpolation group with order unit (see \cite[Chapter 11]{Goodearlbook}).
The space $\Aff T(A)$ can also be viewed as a topological group under addition, equipped with the uniform norm topology.
For unital $A$, the pairing map $\rho_A: K_0(A) \to \Aff T(A)$ is defined as follows: if $x \in K_0(A)$, we can write $x = [p] - [q]$ where $p,q \in M_n(A)$ are projections, and then define
\begin{equation} \rho_A(x)(\tau) := \tr_n \otimes \tau (p - q), \ \ \ \tau \in T(A). \end{equation}
We will denote by $\ku(A),\ka(A)$ the unitary algebraic $K_1$-group and the algebraic $K_1$-group respectively:
\begin{equation} \ku(A) := U_{\infty}(A)/DU_{\infty}(A) \text{ and } \ka(A) := GL_{\infty}(A)/DGL_{\infty}(A). \end{equation}
For the Hausdorffized variants we write
\begin{equation} \hku(A) := U_{\infty}(A)/CU_{\infty}(A) \text{ and } \hka(A) := GL_{\infty}(A)/CGL_{\infty}(A), \end{equation}
where the closures $CU_{\infty}(A) = \ov{DU_{\infty}(A)}$ and $CGL_{\infty}(A) = \ov{DGL_{\infty}(A)}$ are taken with respect to the inductive limit topologies.
That is, $CU_\infty(A) = \dlim\, CU_n(A)$ and $CGL_\infty(A) = \dlim\, CGL_n(A)$. We note that $\hku(A)$ and $\hka(A)$ are topological groups (despite the fact that $U_{\infty}(A)$ and $GL_{\infty}(A)$ are not in general) -- this is addressed in \cite[Remark 2.11]{CGSTW}.

\subsection{The de la Harpe--Skandalis determinant}\label{section:pdak-dlHS-det}

We recall the definition of the de la Harpe--Skandalis determinant \cite{dlHS84a} (see \cite{dlHarpe13} for a more in-depth exposition). Let $\Tr_A: A \to A/\ov{[A,A]}$ be the quotient map from $A$ to the quotient Banach space $A/\ov{[A,A]}$ where $\ov{[A,A]}$ is the closed linear span of additive commutators. We call $\Tr_A$ the universal trace on $A$ and will usually omit the subscript when the C*-algebra is clear from context. For $n \in \bN \cup \{\infty\}$ and a piece-wise smooth path $\xi: [0,1] \to GL_n(A)$, set
\begin{equation} 
\label{eq:predetDefn}
\tD^n(\xi) := \frac{1}{2\pi i} \int_0^1 \Tr(\xi'(t)\xi(t)^{-1})dt.\footnote{We note that when $n=\infty$, by compactness, the image of $\xi$ is contained in some $GL_m(A)$ for $m<\infty$, so we can take the trace.} \end{equation}
We call the map $\tD^n$ the \emph{pre-determinant}. The following properties can be found as \cite[Lemme 1]{dlHS84a}. 

\begin{prop}\label{prop:detFacts}
The map $\tD^n$ which takes a piece-wise smooth path to an element in $A/\ov{[A,A]}$ has the following four properties:
  \begin{enumerate}
    \item it takes pointwise products to sums: if $\xi_1,\xi_2$ are two piece-wise smooth paths, then
    \begin{equation} \tD^n(\xi_1\xi_2) = \tD^n(\xi_1) + \tD^n(\xi_2), \end{equation}
    where $\xi_1\xi_2$ is the piece-wise smooth path $t \mapsto \xi_1(t)\xi_2(t)$ from $\xi_1(0)\xi_2(0)$ to $\xi_1(1)\xi_2(1)$;
  \item if $\|\xi(t) - 1\| < 1$ for all $t \in [0,1]$, then
  \begin{equation} 2\pi i \tD^n(\xi) = \Tr\big(\log(\xi(1)) - \log\xi(0)\big);
  \end{equation}
    \item $\tD^n(\xi)$ depends only on the homotopy class of $\xi$;
    \item if $p \in M_n(A)$ is an idempotent, then the path $\xi_p: [0,1] \to GL_n^0(A)$ given by $\xi_p(t) := pe^{2\pi i t} + (1-p)$ satisfies $\tD^n(p) = \Tr(p)$. \label{predetprojection}
  \end{enumerate}
\end{prop}

  The de la Harpe--Skandalis determinant (at the $n^{\text{th}}$ level) is then the map
  \begin{equation} \Delta^n: GL_{\infty}^0(A) \to \left(A/\ov{[A,A]}\right)/\tD^n(\pi_1(GL_n^0(A))) \end{equation}
  given by $\Delta^n(x) := [\tD^n(\xi_x)]$ where $\xi_x$ is any piece-wise smooth path $\xi_x: [0,1] \to GL_n^0(A)$ from 1 to $x$. This is a group homomorphism to an abelian group, and therefore factors through the derived group, i.e., $DGL_n^0(A) \subseteq \ker \Delta^n$. For the case $n = \infty$, we just write $\tD$ and $\Delta$ for $\tD^{\infty}$ and $\Delta^{\infty}$ respectively. If the C*-algebra needs to be specified, we will write $\Delta^n_A$ or $\Delta_A$.

    \begin{remark}\label{rem:cts-to-piece-wise-smooth}\mbox{}
    \begin{enumerate}
    \item Every continuous path $[0,1] \to GL_n(A)$ is homotopic to a piece-wise smooth path (even a piece-wise smooth exponential path if we are in $GL_n^0(A)$), and as $\tD^n$ is homotopy-invariant, it makes sense to apply $\tD^n$ to any continuous path. Indeed, as in the proof of \cite[Lemme 3]{dlHS84a}, take any continuous path $\xi: [0,1] \to GL_n(A)$ and choose $k$ such that
    \begin{equation}
    \|\xi(\frac{j-1}{k})^{-1}\xi(\frac{j}{k}) - 1\| < 1 \text{ for all } j=1,\dots,k.
    \end{equation}
    Then taking $a_j := \frac{1}{2\pi i}\log \big(\xi(\frac{j-1}{k})^{-1}\xi(\frac{j}{k})\big), j=1,\dots,k$, $\xi$ will be homotopic to the path
    \begin{equation}
    \eta(t) = \xi\left(\frac{j-1}{k}\right)e^{(kt - j)a_j}, t \in \left[\frac{j-1}{k},\frac{j}{k}\right].
    \end{equation}
    We note that if $a = \sum_{j=1}^k a_j$, then $\tD^n(\xi) = \tD^n(\eta) = \Tr(a)$. If $\xi$ is a path of unitaries, then so is $\eta$ and the $a_j$'s are self-adjoint.

    \item   Unless we make any regularity assumptions, the maps $\Delta^n$ may have different codomains as $n$ varies since the images $\tD^n(\pi_1(GL_n^0(A)))$ may vary.
We do however always have 
  \begin{equation} \tD^n(\pi_1(GL_n^0(A))) \subseteq \tD^{n+1}(\pi_1(GL_{n+1}^0(A))) \end{equation}
  since $\tD^{n+1}(\xi \oplus 1) = \tD^n(\xi)$ whenever $\xi$ is a piece-wise smooth loop in $GL_n^0(A)$. 
  However, when the canonical map $\pi_1(GL^0(A)) \to K_0(A)$ is surjective, we have that $\Delta^n = \Delta|_{GL_n^0(A)}$. 
%If we assume that $\tD^1(GL^0(A)) = \tD(GL_{\infty}^0(A))$, then we can just think of the $n^{\text{th}}$ level determinant as the restriction of the determinant at the $\infty$-level. That is, if $\tD^1(GL^0(A)) = \tD(GL_{\infty}^0(A))$,
%  \begin{equation} \tD^n = \tD|_{GL_n^0(A)}, n \in \bN. \end{equation}
    \end{enumerate}
    \end{remark}

  We note that $\pi_1(GL_{\infty}^0(A)) \simeq K_0(A)$ canonically via the map induced by $[\xi_p] \mapsto p$, where $\xi_p$ is the path in property (\ref{predetprojection}) above, and consequently $\Delta$ can be thought of as a map
  \begin{equation} GL_{\infty}^0(A) \to \left(A/\ov{[A,A]}\right)/\Tr(K_0(A)). \end{equation}

  Let $A_0$ consist of elements $a \in A_{sa}$ satisfying $\tau(a) = 0$ for all $\tau \in T(A)$.
  This is a norm-closed real subspace of $A_{sa}$ such that $A_0 \subseteq \ov{[A,A]}$, and there is an isometric identification $A_{sa}/A_0 \simeq \Aff T(A)$ sending an element $[a]$ to $\widehat{a}$, where $\widehat{a}(\tau) := \tau(a)$.
  Indeed, it is not hard to see that the map $A_{sa}/A_0 \to \Aff T(A)$ given by $[a] \mapsto \hat{a}$ is a well-defined $\bR$-linear map.
  Moreover \cite[Theorem 2.9]{CuntzPedersen79}, together with a convexity argument, gives that this is isometric identification.
  To see surjectivity, we note that the image of this map contains constant functions and separates points, so \cite[Corollary 7.4]{Goodearlbook} gives that the image is dense, and therefore all of $\Aff T(A)$ (since this is an isometry).
  We freely identify $A_{sa}/A_0$ with $\Aff T(A)$.

  \begin{lemma}
    Let $A$ be a unital C*-algebra.
The canonical map $\Theta: \Aff T(A) \to A/\ov{[A,A]}$ given by $\Theta(\hat{a}) := [a]$ is an $\bR$-linear isometry. 
    \end{lemma}
        \begin{proof}
          Identifying $\Aff T(A)$ with $A_{sa}/A_0$, $\Theta$ is the map $[a]_{A_{sa}/A_0} \mapsto [a]_{A/\ov{[A,A]}}$. Clearly this is $\bR$-linear if it is well-defined, and it is well-defined since $A_0 \subseteq \ov{[A,A]}$. To show that it's isometric, we have the following chain of inequalities. For $a \in A_{sa}$,
            \begin{equation}
\begin{split}
              \sup_{\tau \in T(A)}|\tau(a)| &\leq \sup_{\tau \text{ s.a, tracial, }\|\tau\| = 1}|\tau(a)| \\
                                            &\leq \sup_{\tau \text{ tracial, }\|\tau\| = 1} |\tau(a)| \\
                                            &= \inf_{x \in \ov{[A,A]}}\|a + x\| \\
                                            &\leq \inf_{x \in A_0} \|a + x\| \\
                                            &= \sup_{\tau \in T(A)}|\tau(a)|,
            \end{split}
\end{equation}
            where the inequalities are obvious, %(2) follows from the Jordan decomposition of self-adjoint tracial functionals (Proposition 2.8 of \cite{CuntzPedersen79})\footnote{Namely, every self-adjoint tracial functional $\tau \in A^*$ can be written as $\tau = \tau_+ - \tau_-$ where $\tau_+,\tau_- \in A^*$ are positive tracial functionals with $\tau_+\tau_- = 0$ and $\|\tau\| = \|\tau_+\| = \|\tau_-\|$.},
            the equality on the third line comes from a standard Hahn--Banach argument, and the last equality comes from our isometric identification $A_{sa}/A_0 \simeq \Aff T(A)$. 
        \end{proof}

        Consequently, we can think of the map $\Tr|_{A_{sa}}$ as the map from $A_{sa} \to \Aff T(A)$ ($\simeq A_{sa}/A_0$) given by $a \mapsto \hat{a}$. 
%We will also denote the map $\Theta: \Aff T(A) \to A/\ov{[A,A]}$ we get from making the identification $A_{sa}/A_0 \simeq \Aff T(A)$. With this, we have for $x \in K_0(A)$,
%    \begin{equation} \Tr(x) = \rho_A(x) \in \Aff T(A) \end{equation}
%    by (\ref{predetprojection}). 

      \begin{lemma}
Let $A$ be a unital C*-algebra.
        For $n \in \bN \cup \{\infty\}$,
        \begin{equation} \tD^n(\pi_1(GL_n^0(A))) = \tD^n(\pi_1(U_n^0(A))) \subseteq \Theta(\Aff T(A)). \end{equation} 
      \end{lemma}
          \begin{proof}
            As $U_n(A)$ is a retract of $GL_n(A)$, the first equality is clear by Proposition \ref{prop:detFacts}(3). Now suppose we have a piece-wise smooth loop $\xi: [0,1] \to U_n^0(A)$. By \cite[Proposition 1.4]{Phillips92},  $\xi'(t)\xi(t)^{-1}$ is skew-adjoint so that $\frac{1}{2\pi i}\xi'(t)\xi(t)^{-1}$ is self-adjoint. Therefore $\Tr(\frac{1}{2\pi i} \xi'(t)\xi(t)^{-1}) \in \Theta(\Aff T(A))$ and, since $\Theta(\Aff T(A))$ is a closed real subspace,
            \begin{equation} \tD^n(\xi) = \int_0^1\Tr\left(\frac{1}{2\pi i}\xi'(t)\xi(t)^{-1} \right) dt \in \Theta(\Aff T(A)) \end{equation}
            as well. 
          \end{proof}

      \begin{cor}
Let $A$ be a unital C*-algebra.
        For $n \in \bN \cup \{\infty\}$ and $u \in U_n^0(A)$,
        \begin{equation} \Delta^n(u) \in \Theta(\Aff T(A))/\tD^n(\pi_1(U_n^0(A))). \end{equation}
      \end{cor}

For $[x] \in A/\ov{[A,A]}$, we'll write
\begin{equation} \text{Re}([x]) := [\text{Re}(x)] = \Theta \left(\widehat{\frac{x + x^*}{2}}\right) \in \Theta(\Aff T(A)), \end{equation}
which is well-defined since $\ov{[A,A]}$ is closed under adjoints.
    Note that $\text{Re}(i [x]) = 0$ if $[x] \in \Theta(A_{sa}/A_0)$, and therefore $\text{Re}(i\Delta^n(\cdot)):GL_n^0(A) \to A_{sa}/A_0$ is well-defined. With this, we have the following fact:
      \begin{equation}
\label{eq:RealPart}
      \text{Re}(2\pi i \Delta^n (x)) = 2\pi i\Delta^n(|x|) = [\log|x|].
\end{equation}
To see this, let $\xi_0: [0,1] \to U_n^0(A)$ be any path from 1 to $u_x$, the unitary part in the polar decomposition of $x$, and let $\xi_1: [0,1] \to GL_n^0(A)$ be the path from 1 to $|x|$ given by $t \mapsto e^{2\pi i t \log|x|}$. Then $\Delta^n(x)$ is the class of $\tD^n(\xi_0) + \tD^n(\xi_1)$ mod $\tD(\pi_1(GL_n^0(A)))$ (which is contained in $\Theta(\Aff T(A))$. As $\tD^n(\xi_0) \in \Theta(\Aff T(A))$, $\text{Re}(2\pi i \tD^n(\xi_0)) = 0$, leaving $\text{Re}(2\pi i \tD^n(\xi_1)) = 2\pi i \tD^n(\xi_1)$. Moreover, $2\pi i \tD^n(\xi_1)$ is clearly equal to $\Theta(\widehat{\log|x|})$ by \eqref{eq:predetDefn}.

      \subsection{Thomsen's variant} Thomsen's variant of the de la Harpe--Skandalis determinant is the Hausdorffized version, taking into account the closure of the image of the homotopy groups. We consider the map
      \begin{equation} \bar{\Delta}^n: GL_n^0(A) \to \left(A/\ov{[A,A]}\right)/\ov{\tD^n(\pi_1(GL_n^0(A)))}, \end{equation}
      given by $\bar{\Delta}^n(x) := [\tD^n(\xi_x)]$ where $\xi_x: [0,1] \to GL_n^0(A)$ is any piece-wise smooth path from 1 to $x \in GL_n^0(A)$.
This is almost the same map as $\Delta^n$, except the codomain is now the quotient by the closure of the image of the fundamental group under the pre-determinant (i.e., the Hausdorffization of the codomain).
Unlike with $\Delta^n$, \cite[Lemma 3.1]{Thomsen95} gives that the kernel of $\bar\Delta^n$ can be identified, without any regularity assumptions on the C*-algebra. 

      \begin{lemma}\label{lem:pdak-comm-int-unitary}
        Let $A$ be a unital C*-algebra. 
            \begin{enumerate}
              \item $\ker \bar\Delta^n|_{U_n^0(A)} = CU_n^0(A)$;
              \item $\ker \bar\Delta^n = CGL_n^0(A)$.
            \end{enumerate}
        \end{lemma}

        We note that \cite[Lemma 3.1]{Thomsen95} only proves (1) above. However working with exponentials $e^a$ with $a \in A$ instead of $e^{ia}$ for $a \in A_{sa}$ yields (2). 

        Two things follow for free here: the first is that $CGL_n^0(A) \cap U_n^0(A) = CU_n^0(A)$ (the inclusion $\supseteq$ is automatic, while $\subseteq$ follows from (1)).
        The second is that the canonical map $U_n^0(A)/CU_n^0(A) \to GL_n^0(A)/CGL_n^0(A)$ is an injection for $n \in \bN \cup \{\infty\}$. 
        Thomsen also gave the following unnatural direct sum decomposition of $\hku(A)$ in terms of $K$-theory and traces (the sequence in (\ref{eq:thomsen-splitting}) splits as one can show that $\Aff T(A)/\ov{\rho_A(K_0(A))}$ is a divisible abelian group).

        \begin{theorem}[Corollary 3.3, \cite{Thomsen95}]
        Let $A$ be a unital C*-algebra. 
            There is an exact sequence
            \begin{equation}\label{eq:thomsen-splitting}
             0 \to \Aff T(A)/\ov{\rho_A(K_0(A))} \to \hku(A) \to K_1(A) \to 0,
             \end{equation}
            which splits unnaturally. 
          \end{theorem}
          Indeed, the splitting above is necessarily unnatural as can be seen in \cite[Section 5]{NielsenThomsen96}. These give examples of morphisms which agree on $K$-theory and traces but disagree on $U(A)/CU(A)$.\footnote{$U(A)/CU(A)$ is isomorphic to $\hku(A)$ as a topological group in this case since the C*-algebras in question satisfy certain $K$-theoretic regularity conditions - see Remark \ref{R:nonstablehausdorffizedgroup}.}

   \section{Polar decomposition}\label{section:polar-decomposition}

We produce direct sum decompositions of the algebraic $K_1$-group of a C*-algebra in terms of the unitary algebraic $K_1$-group and traces. We provide Hausdorffized versions as well. We motivate this with the example of the complex numbers.
  \begin{example}
    Let $A = \bC$. Then we have isomorphisms
\begin{equation} \begin{split}
\ka(A) &\simeq \bC^\times, \quad\text{and} \\
\ku(A) &\simeq \mathbb T, 
\end{split} \end{equation}
via the (usual) determinant map, and
\begin{equation} \Aff T(A) \simeq \bR, \end{equation}
since $A$ has a unique trace.
Hence we see that $\ka(A) \simeq \bT \oplus \Aff T(A)$.
The projection $\ka(A) \to \Aff T(A)$ is given by the canonical map $\log(|\cdot|):\bC^\times \to \bR$, so it is the map
\begin{equation} [x] \mapsto \log(|\det(x)|) = \log(\det(|x|)) = \tr(\log(|x|)),\end{equation}
where $\tr$ is the unnormalized trace on $M_{\infty}(A)$ (which agrees with $tr_n$ if $x \in M_n(A)$). 
  \end{example}

  \subsection{Non-Hausdorffized algebraic $K$-theory}\label{nonhausdorff}
  We start by examining the structure of $U_n^0(A)/\ker \Delta^n|_{U_n^0(A)}$ and $GL_n^0(A)/\ker \Delta^n$. We'll then apply these results to C*-algebras satisfying
  \begin{equation} DU_{\infty}^0(A) = \ker \Delta|_{U_{\infty}^0(A)} \text{ and } DGL_{\infty}^0(A) = \ker \Delta. \end{equation}
First we show that $\Delta^n$ is invariant under conjugation by elements of $GL_n(A)$.
The following is not obvious from the fact that $\Delta^n$ is a homomorphism since $\Delta^n(s)$ is not defined when $s\not\in GL_n^0(A)$.

    \begin{lemma}
        Let $A$ be a unital C*-algebra. 
      For $x \in GL_n^0(A), \Delta^n(s^{-1}xs) = \Delta^n(x)$ for any $s \in GL_n(A)$. 
    \end{lemma}
        \begin{proof}
          If $\xi: [0,1] \to GL_n^0(A)$ is a piece-wise smooth path from 1 to $x$, then $\eta := s^{-1}\xi(\cdot)s: [0,1] \to GL_n^0(A)$ is a piece-wise smooth path from 1 to $s^{-1}xs$ with $\eta'(t) = s^{-1}\xi'(t)s$ (whenever we can differentiate). Consequently (using \eqref{eq:predetDefn}), 
          \begin{equation} \tD^n(\eta) = \tD^n(\xi), \end{equation}
and the result follows.
        \end{proof}

        \begin{lemma}
        Let $A$ be a unital C*-algebra. 
          For $n \in \bN \cup \{\infty\}, x,y \in GL_n(A)$,
          \begin{equation}
\label{eq:DeltaAbsMultplicative}
 \Delta^n(|xy|) = \Delta^n(|x|) + \Delta^n(|y|). \end{equation}
        \end{lemma}
            \begin{proof}
Let $xy=u|xy|$, $x=u_x|x|$, and $y=u_y|y|$ be polar decompositions.
Then $|u^*xy|=|xy|=u^*xy$, and in $GL_n(A)/DGL_n(A)$, we have
\begin{equation} [|xy|]=[u^*xy] = [u^*u_x|x|u_y|y|]=[u^*u_xu_y]+[|x||y|]. \end{equation}
Hence by the previous lemma and using the fact \eqref{eq:RealPart} that $i\Delta^n(|z|)=\text{Re}(i \Delta^n(z))$ for $z \in GL_n^0(A)$,
                \begin{equation}\begin{split}
i\Delta^n(|xy|)&=\text{Re}(i\Delta^n(|xy|)) \\
&=\text{Re}\big(i\Delta^n(u^*u_xu_y)+\Delta^n(|x||y|)\big) \\
&= 0+i\Delta^n(|x|)+i\Delta^n(|y|),
\end{split}\end{equation}
as desired.
            \end{proof}

          \begin{lemma}
        Let $A$ be a unital C*-algebra. 
            The map
          \begin{equation} \chi_n: GL_n(A)/\ker\Delta^n \to \Aff T(A) \end{equation}
          defined by $[x] \mapsto \widehat{\log|x|}$ is a well-defined surjective group homomorphism.
          \end{lemma}
              \begin{proof}
                With our identification $A_{sa}/A_0 = \Aff T(A)$, it is enough to show that
                \begin{equation}
                GL_n(A)/\ker \Delta^n \to A_{sa}/A_0:x  \mapsto [\log|x|]
                \end{equation}
                is well-defined. Let $x,y \in GL_n(A)$ with $x = yz$ for some $z \in \ker \Delta^n$. Now we have
\begin{equation}                  \begin{array}{rcl}
                    [\log|x|] &\stackrel{\eqref{eq:RealPart}}=& 2\pi i \Delta^n(|x|) \\
                                 &=& 2\pi i \Delta^n(|yz|) \\
                                 &\stackrel{\eqref{eq:DeltaAbsMultplicative}}=& 2\pi i \Delta^n(|y|) + 2\pi i \Delta^n(|z|) \\
                                 &\stackrel{\eqref{eq:RealPart}}=& [\log|y|] + \text{Re}(2\pi i \Delta^n(z)) \\
                                 &=& [\log|y|].
\end{array}\end{equation}
                The fact that $\chi_n$ is additive also follows from the previous lemma. Finally, $\chi_n$ is surjective since, for $a \in A_{sa}$, 
\begin{equation} \chi_n([e^a])=[a]. \end{equation}
              \end{proof}

      We'll write $\chi_n^0$ for $\chi_n|_{GL_n^0(A)/\ker \Delta^n}$.

      \begin{lemma}
        Let $A$ be a unital C*-algebra. 
        The map
        \begin{equation} s_n: \Aff T(A) \to GL_n^0(A)/\ker \Delta^n \end{equation}
        given by $\hat{a} \mapsto [e^a]$ is a well-defined group homomorphism.
          \begin{proof}
Again, we'll identify $\Aff T(A) = A_{sa}/A_0$. If $[a] = [b]$ in $A_{sa}/A_0$, then
            \begin{equation} \Delta^n(e^a) = \frac{1}{2\pi i}\Tr(a) = \frac{1}{2\pi i}\Tr(b) = \Delta^n(e^b), \end{equation}
            giving well-definedness. Moreover, for $[a],[b]\in A_{sa}/A_0$,
              \begin{equation} \begin{split}
                \Delta^n(e^{a+b}) &= \frac{1}{2\pi i}\Tr(a+b) \\
                                &= \frac{1}{2\pi i}\Tr(a) + \frac{1}{2\pi i}\Tr(b) \\
                                &= \Delta^n(e^a) + \Delta^n(e^b). 
              \end{split} \end{equation}
Hence, $[a^{a+b}]=[e^ae^b]$ in $GL_n^0/\ker\Delta^n$.
          \end{proof}
      \end{lemma}

        \begin{remark}
          Now we explain why we are working with $\Tr$ instead of working with each tracial state concurrently. If we worked with $\Delta_\tau$, where $\tau$ ranges over $\tau \in T(A)$, the same arguments above will hold. However, unless one makes a separability assumption (more specifically, that $K_0(A)$ is countable), we don't necessarily have $\ker\Delta = \bigcap_{\tau \in T(A)} \ker\Delta_\tau$. Indeed, if we had a piece-wise smooth path $\xi_\tau$ from 1 to $x$ with $\tD_\tau(\xi_\tau) \in \tau(K_0(A))$ for all $\tau \in T(A)$, it is not necessarily true that we can find a single element of $x \in K_0(A)$ such that $\tD_\tau(\xi_\tau) = \rho_A(x)(\tau)$ for all $\tau \in T(A)$. See Lemme 5 and Proposition 6 of \cite{dlHS84a}. 
        \end{remark}

        \begin{prop}
        Let $A$ be a unital C*-algebra. 
          The sequence
          \begin{equation} \begin{tikzcd}
0 \arrow[r] & U_n^0(A)/ker \Delta^n|_{U_n^0(A)} \arrow[r, "\iota_n^0"] & GL_n^0(A)/\ker\Delta^n \arrow[r, "\chi_n^0"] & \Aff T(A) \arrow[r] \arrow[l, "s_n"', bend right=69] & 0
\end{tikzcd} \end{equation}
is a split short exact sequence.
        \end{prop}
  \begin{proof}
We know that $\iota_n^0$ is injective and that $\chi_n^0\circ s_n = \mathrm{id}_{\Aff(T(A))}$ (and hence $\chi_n^0$ is surjective). We must show that
    \begin{equation} \ker \chi_n^0 = \mathrm{Im}(\iota_n^0). \end{equation}
    The containment $\supseteq$ is trivial since a unitary has positive part equal to 1. For the reverse containment, suppose that  $x \in GL_n^0(A)$ is such that $\Tr(\log|x|) = 0$. Then, letting $x=u_x|x|$ be the polar decomposition,
      \begin{equation} \begin{split}
      \Delta^n(x) &= \Delta^n(u_x) + \Delta^n(|x|) \\
                  &= \Delta^n(u_x) + \frac{1}{2\pi i}\Tr(\log |x|) \\
                  &= \Delta^n(u_x) \in \iota_n^0\Big(U_n^0(A)/\ker \Delta^n|_{U_n^0(A)}\Big).
      \end{split} \end{equation}
  \end{proof}

        We wish to remove the superscript 0 to get a sequence involving
        $ U_n(A)/\ker\Delta^n|_{U_n^0(A)}$, and $GL_n(A)/\ker\Delta^n, A_{sa}/A_0$

        \begin{theorem}\label{glnsplitting}
        Let $A$ be a unital C*-algebra. 
        The sequence 
          \begin{equation} \begin{tikzcd}
0 \arrow[r] & U_n(A)/ker \Delta^n|_{U_n^0(A)} \arrow[r, "\iota_n"] & GL_n(A)/\ker\Delta^n \arrow[r, "\chi_n"] & \Aff T(A) \arrow[r] \arrow[l, "s_n"', bend right=69] & 0
\end{tikzcd} \end{equation}
is a split short exact sequence.
\end{theorem}
  \begin{proof}
    We have the following commutative diagram:
    \begin{equation} \begin{tikzcd}
            & 0 \arrow[d]                                                         & 0 \arrow[d]                                             & 0 \arrow[d]                                         &   \\
0 \arrow[r] & U_n^0(A)/\ker \Delta^n|_{U_n^0(A)} \arrow[d] \arrow[r, "\iota_n^0"] & GL_n^0(A)/\ker \Delta^n \arrow[d] \arrow[r, "\chi_n^0"] & \Aff T(A) \arrow[d, equal] \arrow[r] & 0 \\
0 \arrow[r] & U_n(A)/\ker \Delta^n|_{U_n^0(A)} \arrow[d] \arrow[r, "\iota_n"]      & GL_n(A)/\ker \Delta^n \arrow[d] \arrow[r, "\chi_n"]     & \Aff T(A) \arrow[d] \arrow[r]                      & 0 \\
0 \arrow[r] & \pi_0(U_n(A)) \arrow[d] \arrow[r, "\simeq"]                         & \pi_0(GL_n(A)) \arrow[d] \arrow[r]                      & 0 \arrow[d] \arrow[r]                               & 0 \\
            & 0                                                                   & 0                                                       & 0                                                   &  
\end{tikzcd} \end{equation}
where all the columns, as well as the 1$^\text{st}$ and 3$^\text{rd}$ rows are exact. As we have
\begin{equation} \iota_n(U_n(A)/\ker\Delta^n|_{U_n^0(A)}) \subseteq \ker \chi_n, \end{equation}
  it follows from \cite[Exercise II.5.2]{Maclane12} (or a diagram chase) that the second row is also exact. It is easy to see that $s_n: \Aff(T) \to GL_n^0(A)/\ker \Delta^n \subseteq GL_n(A)/\ker\Delta^n$ is also a splitting for the second row as
  \begin{equation} \chi_n \circ s(\widehat{a}) = \chi_n([e^a]) = \widehat{\log |e^a|} = \widehat{a}. \end{equation}
  \end{proof}

\begin{cor}\label{Ktheorysplitting}
          Suppose that $A$ is a unital C*-algebra and let $n \in \bN$. If $\ker \Delta^n|_{U_n^0(A)} = DU_n^0(A)$ and $\ker \Delta^n = DGL_n^0(A)$, then 
          \begin{equation} \begin{tikzcd}
0 \arrow[r] & U_n(A)/DU_n^0(A) \arrow[r, "\iota_n"] & GL_n(A)/DGL_n^0(A) \arrow[r, "\chi_n"] & \Aff T(A)  \arrow[r] \arrow[l, "s_n"', bend right=79] & 0
\end{tikzcd} \end{equation}
          is a split short exact sequence. In particular, for $n = \infty$, we have a natural split short exact sequence
          \begin{equation} \begin{tikzcd}
0 \arrow[r] & \ku(A) \arrow[r, "\iota"] & \ka(A) \arrow[r, "\chi_\infty"] & \Aff T(A) \arrow[r] \arrow[l, "s"', bend right=79] & 0.
\end{tikzcd} \end{equation}
        \end{cor}
  \begin{proof}
    The first part follows from the above as $DU_n^0(A) = \ker \Delta^n|_{U_n^0(A)}$ and $DGL_n^0(A) = \ker \Delta^n$. For the last part, if $n = \infty$, then $DGL_{\infty}(A) = DGL_{\infty}^0(A)$ by Whitehead's lemma. Indeed, if $x \in GL_n(A)$ is a commutator, say $x = yzy^{-1}z^{-1}$, then $x \oplus 1 \oplus 1 \in GL_{3n}(A)$ can be written as a commutator as follows:
   \begin{equation} \label{eq:3x3-trick}
      \begin{pmatrix}
      x \\ & 1 \\ & & 1
      \end{pmatrix} =   \begin{pmatrix}
      y \\ & y^{-1} \\ & & 1
        \end{pmatrix}   \begin{pmatrix}
      z \\ & 1 \\ & & z^{-1}
          \end{pmatrix}   \begin{pmatrix}
      y^{-1} \\ & y \\ & & 1
            \end{pmatrix}   \begin{pmatrix}
      z^{-1} \\ & 1 \\ &  & z
    \end{pmatrix}. \end{equation}
    The four matrices on the right are connected to the identity by Whitehead's lemma (see \cite[Lemma 2.1.5]{RordamKBook}). 
  \end{proof}

  The above split exact sequence yields that
    \begin{equation}
      \label{eq:alg-K-theory-splitting}     \ka(A) \simeq \ku(A) \oplus \Aff T(A)
    \end{equation}
    naturally via the isomorphism
      \begin{equation}
        [x] \mapsto [u_x] \oplus \widehat{\log|x|}.
      \end{equation}
  The following is an immediate consequence.

            \begin{cor}
              Let $A,B$ be unital C*-algebras such that
                  \begin{equation} DU_{\infty}^0(A) = \ker \Delta|_{U_{\infty}^0(A)} \text{ and }DGL_{\infty}^0(A) = \ker \Delta. \end{equation}
                If $x,y \in GL_{\infty}(A)$, the following are equivalent. 
                  \begin{enumerate}
                    \item $[u_x] = [u_y]$ in $\ku(A)$ and $\widehat{\log|x|} = \widehat{\log|y|}$ in $\Aff T(A)$;
                    \item $[x] = [y]$ in $\ka(A)$. 
                  \end{enumerate}
            \end{cor}

        For $\phi: A \to B$ a *-homomorphism between unital C*-algebras, denote by
          \begin{enumerate}
            \item $\ku(\phi):\ku(A) \to \ku(B)$;
            \item $\ka(\phi): \ka(A) \to \ka(B)$;
            \item $T(\phi): T(B) \to T(A)$
          \end{enumerate}
          the maps induced by $\phi$.

            \begin{cor}
              Let $A,B$ be unital C*-algebras such that
                \begin{itemize}
                  \item $DU_{\infty}^0(A) = \ker \Delta_A|_{U_{\infty}^0(A)}$ and $DGL_{\infty}^0(A) = \ker \Delta_B$;
                  \item $DU_{\infty}^0(B) = \ker \Delta_B|_{U_{\infty}^0(A)}$ and $DGL_{\infty}^0(B) = \ker \Delta_B$.
                \end{itemize}
                Let $\phi,\psi: A \to B$ be unital *-homomorphisms. The following are equivalent.
                  \begin{enumerate}
                    \item $\ku(\phi) = \ku(\psi)$ and $T(\phi) = T(\psi)$;
                    \item $\ka(\phi) = \ka(\psi)$.
                  \end{enumerate}
            \end{cor}

            There are many classes of unital C*-algebras satisfying the two hypotheses of the above corollary \cite{dlHS84b,Thomsen93,Ng14,NgRobert17,NgRobert15}. In the penultimate, it is shown there that the hypotheses hold in the case that $A$ is a unital, separable, simple, pure C*-algebra of stable rank one such that every 2-quasitracial state is a trace, and the in the latter it is shown to hold  for $M_3(A)$ whenever $A$ is pure \cite[Definition 3.6]{Winter12}.

      \begin{cor}\label{cor:pdak-examples-of-ker}
        Let $A$ be a unital, simple, separable, pure C*-algebra of stable rank one such that every 2-quasitrace is a trace. Then there is a natural isomorphism
        \begin{equation} \ka(A) \simeq \ku(A) \oplus \Aff T(A). \end{equation}
      \end{cor}

\subsection{Hausdorffized algebraic $K$-theory}

In the Hausdorffized setting, we obtain similar results by the same arguments. However, in this case, we have $\ker\bar{\Delta}^n|_{U_n^0(A)} = CU_n^0(A)$ and $\ker\bar{\Delta}^n = CGL_n^0(A)$ by Lemma \ref{lem:pdak-comm-int-unitary}. 
Let
\begin{equation}\label{eq:hausdorffized-maps}  
\begin{split}
    \bar{\iota}_n&:U_n(A)/CU_n^0(A) \to GL_n(A)/CGL_n^0(A), \\
    \bar{\chi}_n&: GL_n(A)/CGL_n^0(A) \to \Aff T(A),\quad\text{and} \\
    \bar{s}_n&: \Aff T(A) \to GL_n(A)/CGL_n^0(A)
  \end{split} \end{equation}
  be the variants of the maps $\iota_n,\chi_n,s_n$ in the previous section (so our domains and codomains are now topological).
  Identifying $CU_n^0(A) = \ker\ov\Delta_n$ and applying the arguments from Section \ref{nonhausdorff} gives that each of these maps are well-defined group homomorphisms for $n \in \bN \cup \{\infty\}$. 
  In the Hausdorffized setting, we show that these maps are also continuous. First a lemma to handle the $n = \infty$ case.

  \begin{lemma}\label{lem:open-quotient-map}
    Let $G = \cup_n G_n$ be an increasing union of topological groups and equip $G$ with the inductive limit topology. Let $H \leq G$ be a subgroup such that the closure $CH$ of $H$ is also a subgroup of $G$. Then the quotient map $q: G \to G/CH$ is an open map. 
      \begin{proof}
        Let $S \subseteq G$ be open. As $G/CH$ has the quotient topology, the set $q(S) \subseteq G/CH$ is open if and only if $q^{-1}(q(S)) \subseteq G$ is open in $G$. Thinking of $G/CH$ as the space of $CH$-orbits of $G$ where $CH \act G$ by right translation, we have that
          \begin{equation}
            q^{-1}(q(S)) = \bigcup_{h \in CH} Sh
          \end{equation}
          which is open if $S$ is, since right translation still yields a homeomorphism in the inductive limit topology -- see \cite[Proposition 1.1(ii)]{TatsuumaShimomuraHirai98}. 
      \end{proof}
  \end{lemma}

    \begin{prop}\label{prop:cts-and-open}
    The maps in (\ref{eq:hausdorffized-maps}) are well-defined, continuous group homo-morphisms.
    Moreover, $\ov\iota$ and $\ov\chi$ are open onto their images. 
    \begin{proof}
      A straightforward adaptation of the arguments of the previous section shows that these are well-defined group homomorphisms.
      We work with the $n = \infty$ case throughout, as the $n \in \bN$ case is similar, and easier due to the fact that $GL_n(A),U_n(A)$ are topological groups.

      Let us show that $\ov\iota$ is continuous. The diagram
        \begin{equation}
          \begin{tikzcd}
            U_\infty(A) \arrow[r, "\sigma"] \arrow[d , "q_U"'] & GL_\infty(A) \arrow[d, "q_{GL}"] \\
\hku(A) \arrow[r, "\ov\iota"']   & \hka(A) 
\end{tikzcd}
        \end{equation}
        commutes where the left and right maps are quotient maps and $\sigma$ is the canonical inclusion. We note that for any subset $S \subseteq \hka(A)$ the commutation of the above diagram gives that
        \begin{equation}
          (q_U)^{-1}\left(\ov\iota^{-1}(S)\right) = \sigma^{-1}\left(q_{GL}^{-1}(S)\right).
        \end{equation}
        Therefore if $S \subseteq \hka(A)$ is open, then
        \begin{equation}\label{eq:oviota-diagram}
            \begin{split}
            \ov\iota^{-1}(S) &= q_U\left(q_U^{-1}\left( \ov\iota^{-1}(S)\right)\right) \\
                             &= q_U\left(\sigma^{-1}\left( q_{GL}^{-1}(S)\right)\right),
            \end{split}
          \end{equation}
          where $\sigma^{-1}\left(q_{GL}^{-1}(S)\right)$ is open because both $q_{GL}$ and $\sigma$ are continuous. As $q_U$ is open by Lemma \ref{lem:open-quotient-map}, it follows that $\ov\iota^{-1}(S)$ is open. This shows continuity.

        Let us show that $\ov\iota$ is open onto its image. We note that taking the unitary part of the polar decomposition $\omega_n: GL_n(A) \to U_n(A) \subseteq U_{\infty}(A)$ is continuous for all $n$ and therefore induces a continuous map $\omega: GL_{\infty}(A) \to U_{\infty}(A)$. Since $CGL_{\infty}(A) \cap U_{\infty}(A) = CU_{\infty}(A)$ by Lemma \ref{lem:pdak-comm-int-unitary}, we get an induced continuous map
\begin{equation}
        \ov \omega: \hka(A) \to \hku(A)
\end{equation}
        which clearly satisfies
\begin{equation}
\ov\omega \circ \ov\iota = \id_{\hku(A)} \text{ and } \ov\iota \circ \ov\omega|_{\ov\iota\left(\hku(A)\right)} = \id_{\ov\iota\left(\hku(A)\right)}.
\end{equation}        
        Now if $S \subseteq \hku(A)$ is open, then it is easily seen that
        \begin{equation}
        \ov\iota(S) = \ov\iota\left(\hku(A)\right) \cap ({\ov\omega})^{-1}(S).
        \end{equation}
        As $\ov\omega$ is continuous, $(\ov\omega)^{-1}(S) \subseteq \hka(A)$ is open and so $\ov\iota(S) \subseteq \ov\iota\left(\hku(A)\right)$ is open with respect to the subspace topology. This shows that $\ov\iota$ is open onto its image. 

          For $\ov\chi$, let $g: GL_{\infty}(A) \to \Aff T(A)$ denote the map $g(x) := \widehat{\log|x|}$. The diagram
            \begin{equation}\label{eq:showing-ovchi-is-open-g}
              \begin{tikzcd}
GL_{\infty}(A) \arrow[rr, "g"] \arrow[rd, "q_{GL}"'] &                               & \Aff T(A) \\
                                                & \hka(A) \arrow[ru, "\ov\chi"'] &          
\end{tikzcd}
            \end{equation}
            commutes, so we have that for $S \subseteq \Aff T(A)$
              \begin{equation}
                  \begin{split}
                    \ov\chi^{-1}(S) &= q_{GL}\left(q_{GL}^{-1}\left(\ov\chi^{-1}(S)\right)\right) \\
                                    &= q_{GL}\left(g^{-1}(S)\right).
                  \end{split}
              \end{equation}
              Thus since we know that $q_{GL}$ is open by Lemma \ref{lem:open-quotient-map}, it suffices to show that $g$ is continuous. But $g$ is continuous if $g|_{GL_n(A)}$ is continuous for all $n$,\footnote{If $X = \cup_nX_n$ is equipped with the inductive limit topology and $f: X \to Y$ is a function such that $f|_{X_n}$ is continuous for all $n$, then $f$ is continuous. To see this, let $S \subseteq Y$ be open and note that $f^{-1}(S) = \cup_n f|_{X_n}^{-1}(S)$ is open.} and this is true: indeed, $g|_{GL_n(A)}$ can be written as the composition
                \begin{equation}
                  \begin{tikzcd}
G_n(A) \arrow[r, "l|_{GL_n(A)}"] & A_{sa} \arrow[r, "\Tr|_{A_{sa}}"] & \Aff T(A)
\end{tikzcd}
                \end{equation}
                where $l: GL_{\infty}(A) \to A_{sa}$ is the map given by $l(x) := \tr \log|x|$ where $\tr: M_{\infty}(A) \to A$ is the unnormalized trace. Seeing that $l|_{GL_n(A)}$ is continuous follows easily: if $x_n \to x$ in $GL_n(A)$, then $\tr \log|x_n| \to \tr \log|x|$. 

          To show that $\ov\chi$ is open, we again appeal to the diagram (\ref{eq:showing-ovchi-is-open-g}). It suffices to show that $g$ is open -- and to this end it suffices to show that $g|_{GL_n}$ is open for each $n$.\footnote{If $X = \cup_n X_n$ is an increasing union of topological spaces with the inductive limit topology and $Y$ is another topological space, then for $S \subseteq X$, we have that $f(S) = \cup_n f(S \cap X_n)$.}
  For $GL_n(A)$, $\widehat{\log |x_0|} = \widehat{\tr \log|x_0|}$, where $\tr$ is the unnormalized trace, and so we can restrict to the case where $n = 1$. 
  Let us without loss of generality work with open balls around the identity: let $\ee > 0$ and consider
    \begin{equation}
      S := \{x \in GL(A) \mid \|x - 1\| < \ee\}.
    \end{equation}
  Looking at the image of $S$ under $g_1 :=g|_{GL(A)}$, we have
    \begin{equation}
      g_1(S) = \{\widehat{\log|x|} \mid \|x - 1\| < \ee\}.
    \end{equation}
  Let $x_0 \in GL(A)$ be such that $x_0 \approx_\ee 1$ and let us show that there is an open ball around $\widehat{\log|x_0|}$ that is contained in $g_1(S)$. 
  First off note that for $\hat{h} \in \Aff T(A)$, with $h \in A_{sa}$, we have
  \begin{equation}\label{eq:chi-open}
        \begin{split}
          \|\widehat{\log|x_0|} - \hat{h}\| &= \|\widehat{\log|x_0|} - \widehat{\log|e^h|}\| \\
                                          &= \|\ov\chi([x_0]) - \ov\chi([e^h])\| \\
                                          &= \|\ov\chi([x_0e^{-h}])\| \\
                                          &= \|\widehat{\log|x_0e^{-h}|}\|.
        \end{split}
    \end{equation}
    Now let $\delta > 0$ be such that whenever $a \in A$, we have $\|a\| < \delta$ implies that $\|e^a - 1\| < \ee$.
    Then, for $\hat{h} \in \Aff T(A)$ with $\widehat{\log|x_0|} \approx_{\frac{\delta}{2}} \hat{h}$, we have by (\ref{eq:chi-open}) that
      \begin{equation}
        \|\widehat{\log|x_0e^{-h}|}\| < \frac{\delta}{2}.
      \end{equation}
      Find a self-adjoint lift, say $k \in A_{sa}$, of $\widehat{\log |x_0e^{-h}|}$ with $\|k\| < \delta$.
      Then we have that $\|e^k - 1\| < \ee$ and 
      \begin{equation}
        g_1(e^k) = \hat{k} = \widehat{\log|x_0e^{-h}|}.
      \end{equation}
      This shows that $B_{\frac{\delta}{2}}(\widehat{\log|x_0|}) \subseteq g_1(S)$.

          Finally, let us show that $\ov s$ is continuous. We have that 
            \begin{equation}
              \begin{tikzcd}
A_{sa} \arrow[r, "\alpha"] \arrow[d, "\Tr|_{A_{sa}}"'] & GL_{\infty}(A) \arrow[d, "q_{GL}"] \\
\Aff T(A) \arrow[r, "\ov s"']                          & \hka(A)                           
\end{tikzcd}
            \end{equation}
            commutes where $\alpha(a) := e^a$ -- note that $\alpha$ is continuous and that the image of $\alpha$ is contained in $GL(A) \subseteq GL_{\infty}(A)$ ($\alpha$ is however \emph{not} a homomorphism). Consequently if $S \subseteq \hka(A)$ is open, then since $\Tr|_{A_{sa}}$ is surjective,
              \begin{equation}
                  \begin{split}
                    \ov s^{-1}(S) &= \Tr|_{A_{sa}}\left(\Tr_{A_{sa}}^{-1}\left(\ov s^{-1}(S)\right)\right) \\
                                  &= \alpha^{-1}\left(q_{GL}^{-1}(S)\right)
                  \end{split}
              \end{equation}
              is open as well since $\alpha$ and $q_{GL}$ are continuous. 
    \end{proof}
  \end{prop}

    \begin{theorem}\label{cor:pdak-top-group-splitting}
      For any unital C*-algebra $A$ and $n \in \bN\cup\{\infty\}$, the sequence 
          \begin{equation} \begin{tikzcd}
            0 \arrow[r] & U_n(A)/CU_n^0(A) \arrow[r, "\bar\iota_n"] & GL_n(A)/CGL_n^0(A) \arrow[r, "\bar\chi_n"] & \Aff T(A)  \arrow[r] \arrow[l, "\bar{s}_n"', bend right=79] & 0
\end{tikzcd} \end{equation}
is a split short exact sequence of topological groups. In particular, we have the following split short exact sequence of topological groups:
          \begin{equation} \begin{tikzcd}
            0 \arrow[r] & \hku(A) \arrow[r, "\bar\iota"] & \hka(A) \arrow[r, "\bar\chi_\infty"] & \Aff T(A)  \arrow[r] \arrow[l, "\bar{s}"', bend right=79] & 0
\end{tikzcd} \end{equation}
  \begin{proof}
    The same argument as in Theorem \ref{glnsplitting} gives an algebraic splitting.
    The fact that this is a splitting of topological groups follows as $\ov\iota_n,\ov\chi_n,\ov s_n$ are all continuous and $\ov\iota_n$ and $\ov\chi_n$ are open onto their images by Proposition \ref{prop:cts-and-open}. 
  \end{proof}
    \end{theorem}

        \begin{cor}
          Let $A,B$ be unital C*-algebras, $x,y \in \hka(A)$. The following are equivalent.
            \begin{enumerate}
              \item $[u_x] = [u_y]$ in $\hku(A)$ and $\widehat{\log|x|} = \widehat{\log|y|}$ in $\Aff T(A)$;
              \item $[x] = [y]$ in $\hka(A)$. 
            \end{enumerate}
        \end{cor}

    For $A,B$ unital C*-algebras, $\phi: A \to B$ a unital *-homomorphism, denote by 
      \begin{enumerate}
        \item $\hku(\phi): \hku(A) \to \hku(B)$;
        \item $\hka(\phi): \hka(A) \to \hka(B)$
      \end{enumerate}
      the maps induced by $\phi$.

      \begin{cor}
        Let $A,B$ be unital C*-algebras. Let $\phi,\psi: A \to B$ be unital *-homomorphisms. The following are equivalent.
          \begin{enumerate}
            \item $\hku(\phi) = \hku(\psi)$ and $T(\phi) = T(\psi)$;
            \item $\hka(\phi) = \hka(\psi)$. 
          \end{enumerate}
      \end{cor}

        \section{Nonstable algebraic $K$-theory}\label{section:pdak-nonstable-algebraic-k-theory}

        Here we discuss some structure of the non-stable (both Hausdorffized and not) algebraic $K_1$-groups. In \cite[Theorem 3.2]{Thomsen95}, Thomsen proved that the map
        \begin{equation}
          U_{\infty}^0(A)/CU_{\infty}^0 \simeq \Aff T(A)/\ov{\rho_A(K_0(A))}
        \end{equation}
        given by $[u] \mapsto \ov\Delta(u)$ is a homeomorphic isomorphism. 
        It was noted that if $\pi(U^0(A)) \to K_0(A)$ is surjective, then of course we have that
          \begin{equation}
            U^0(A)/CU^0(A) \simeq U_n^0(A)/CU_n^0(A)
          \end{equation}
          for all $n \in \bN \cup \{\infty\}$.
        Indeed, if the canonical map $\pi_1(U^0(A)) \to K_0(A)$ is surjective, then the following diagram commutes
        \begin{equation}\label{eq:Thomsen-isomorphism}
            \begin{tikzcd}
U^0(A)/CU^0(A) \arrow[d, "\ov D_1"'] \arrow[r, "\ov i"]     & U_{\infty}^0(A)/CU_{\infty}^0(A) \arrow[d, "\ov D"] \\
\Aff T(A)/\ov{\tD(\pi_1(U^0(A)))} \arrow[r, "\id"'] & \Aff T(A)/\ov{\rho_A(K_0(A))}                   
\end{tikzcd}
          \end{equation}
          where $\ov i: U^0(A)/CU^0(A) \to U_{\infty}^0(A)/CU_{\infty}^0(A)$ is the canonical map, $\ov D_1, \ov D$ are maps factoring $\Delta^1$ and $\Delta$ through $CU^0(A)$ and $CU_{\infty}^0(A)$ respectively.
          As $\id,\ov D_1$ and $\ov D$ are all homeomorphic isomorphisms, it follows that the canonical map $\ov i$ is a homeomorphic isomorphism. 

          \begin{remark}
            More generally one can study the question of when 
              \begin{equation}
            U_n^0(A)/CU_n^0(A) \to U_m^0(A)/CU_m^0(A)
              \end{equation}
            is an isomorphism for all $m \geq n$, even in the case where $\pi_1(U^0(A)) \to K_0(A)$ may not be surjective.
            See \cite{GongLinXue15} for details.
            One can of course get similar results using the general linear invariants, as well as the purely algebraic variants under the assumptions that  $\ker\Delta_n|_{U_n^0(A)} = DU_n^0(A)$ or $\ker\Delta_n = DGL_n^0(A)$ for every $n$.
          \end{remark}

          A similar argument gives the following in the algebraic setting. 

          \begin{lemma}\label{lem:nonstable-isomorphism}
            Let $A$ be a unital C*-algebra and suppose that $\pi_1(U^0(A)) \to K_0(A)$ is surjective.
              \begin{enumerate}
                \item The canonical map $U^0(A)/\ker\Delta^1|_{U^0(A)} \to U_{\infty}^0(A)/\ker \Delta|_{U_{\infty}^0(A)}$ is an isomorphism. 
                \item The canonical map $GL^0(A)/\ker \Delta^1 \to GL_{\infty}^0(A)/\ker \Delta$ is an isomorphism.
              \end{enumerate}
                \begin{proof}
                  Writing out a similar diagram to (\ref{eq:Thomsen-isomorphism}), we have
                    \begin{equation}
                      \begin{tikzcd}
                        U^0(A)/\ker\Delta^1|_{U^0(A)} \arrow[d, "D_1"'] \arrow[r, "i"] & U_{\infty}^0(A)/\ker \Delta|_{U_{\infty}^0(A)} \arrow[d, "D"] \\
\Aff T(A)/\tD(\pi_1(U^0(A))) \arrow[r, "\id"'] & \Aff T(A)/\rho_A(K_0(A))                        
\end{tikzcd}
                    \end{equation}
                    The maps $\id,D_1,D$ are all group isomorphisms, so $i$ must be as well (the maps $i,D_1,D$ are the purely algebraic analogues of $\ov i, \ov D_1, \ov D$ above). 
                    We get a similar diagram in the $GL$ setting with $U$ replaced with  $GL$ and $\Aff T(A)$ replaced by $A/\ov{[A,A]}$.
                \end{proof}
          \end{lemma}

          Using similar techniques to \cite{GongLinXue15}, we have the following. 

          \begin{lemma}\label{lem:non-stable-kernel}
              Let $A$ be a unital C*-algebra and suppose that $\pi_1(U^0(A)) \to K_0(A)$ is surjective.
                \begin{enumerate}
                  \item If $\ker \Delta^1|_{U^0(A)} = DU^0(A)$, then $\ker \Delta^n|_{U_n^0(A)} = DU_n^0(A)$ for all $n \in \bN \cup \{\infty\}$.
                  \item If $\ker \Delta^1 = DGL^0(A)$, then $\ker \Delta^n = DGL_n^0(A)$ for all $n \in \bN \cup \{\infty\}$. 
                \end{enumerate}
            \end{lemma}
                \begin{proof}
                  We show (1) holds, (2) is similar. Suppose that $u \in \ker \Delta^n|_{U_n^0(A)}$. There is some $a \in A_{sa}$ such that $[u] = [e^{2\pi i a} \oplus 1_{n-1}]$ and a piece-wise smooth path $\xi:[0,1] \to U_n^0(A)$ with $\tD(\xi) = \Tr(a)$ by Remark \ref{rem:cts-to-piece-wise-smooth}(1).

                  As $u \in \ker\Delta^n|_{U_n^0(A)}$, there is some piece-wise smooth loop $\eta:[0,1] \to U_n^0(A)$ with
                  \begin{equation} \tD^n(\xi) = \tD^n(\eta). \end{equation}
                  As before, the surjectivity of $\pi_1(U^0(A)) \to K_0(A)$ implies that $\eta$ is homotopic to $\eta_0 \oplus 1_{n-1}$ for some piece-wise smooth loop $\eta_0: [0,1] \to U^0(A)$. Then $\eta_1(t) := e^{2\pi i ta}\eta_0(t)^*$ defines a piece-wise smooth path in $U^0(A)$ from 1 to $e^{2\pi i a}$ such that $\tD(\eta_1) = 0$. Therefore $e^{2\pi i a} \in \ker \Delta^1|_{U^0(A)} = DU^0(A)$ and consequently
                  \begin{equation} [u] = [e^{2\pi i a} \oplus 1_{n-1}] = 0 \text{ in  } U_n^0(A)/DU_n^0(A). \end{equation}
                \end{proof}

      Now we finish by showing that we can work outside of the connected component.

        \begin{theorem}
          Let $A$ be a unital C*-algebra such that
            \begin{enumerate}
              \item the canonical map $\pi_1(U^0(A)) \to K_0(A)$ is surjective;
              \item the canonical map $U(A)/U^0(A) \to K_1(A)$ is an isomorphism.
            \end{enumerate}
          Then the following is true.
          \begin{enumerate}
            \item If $\ker\Delta^1|_{U^0(A)} = DU^0(A)$, then $U(A)/DU(A) \simeq \ku(A)$.
            \item If $\ker\Delta^1 = DGL^0(A)$, then $GL(A)/DGL(A) \simeq \ka(A)$.
            \end{enumerate}
        \begin{proof}
          We show this in the unitary setting. 
          First we again note that $DU_{\infty}^0(A) = DU_{\infty}(A)$ by (\ref{eq:3x3-trick}), and so $\ku(A) = U_{\infty}/DU_{\infty}^0(A)$.
          Moreover, (2) implies that $A$ is $K_1$-injective, giving that $DU(A) = DU^0(A)$ and $DGL(A) = DGL^0(A)$ as well.
          Using property (1), together with the fact that $\ker \Delta^1|_{U^0(A)} = DU^0(A)$, gives that the canonical map
          \begin{equation}\label{eq:nonstable-iso}
              U^0(A)/DU^0(A) \simeq U_{\infty}^0(A)/DU_{\infty}^0(A)
            \end{equation}
            is an isomorphism by combining Lemmas \ref{lem:nonstable-isomorphism} and \ref{lem:non-stable-kernel}.
            Now combining (2) with (\ref{eq:nonstable-iso}) yields a morphism of short exact sequences
              \begin{equation}
                \begin{tikzcd}\label{eq:morph-ses}
0 \arrow[r] & U^0(A)/DU(A) \arrow[r] \arrow[d, "\simeq"] & U(A)/DU(A) \arrow[r] \arrow[d] & U(A)/U^0(A) \arrow[r] \arrow[d, "\simeq"] & 0 \\
0 \arrow[r] & U_{\infty}^0(A)/DU_{\infty}(A) \arrow[r]   & \ku(A) \arrow[r]               & K_1(A) \arrow[r]                          & 0
\end{tikzcd}
              \end{equation}
              where the left and right vertical maps are isomorphisms. Therefore the middle vertical map is an isomorphism by the Short Five Lemma \cite[Lemma I.3.1]{Maclane12}.
              The argument in the general linear setting is the same with $U$ replaced by $GL$ and $\ku$ replaced with $\ka$.
        \end{proof}
        \end{theorem}

        \begin{remark}\label{R:nonstablehausdorffizedgroup}
            If $A$ is unital, with the assumptions
              \begin{enumerate}
                \item $\pi_1(U^0(A)) \to K_0(A)$ is surjective and
                \item $U(A)/U^0(A) \to K_1(A)$ is an isomorphism,
              \end{enumerate}
              we also get that
                \begin{equation}
                  U(A)/CU(A) \simeq \hku(A) \text{ and } GL(A)/CGL(A) \simeq \hka(A),
                \end{equation}
                even as topological groups. Indeed, looking at the unitary case for example, since the map
                  \begin{equation}
                    U^0(A)/CU^0(A) \simeq U_{\infty}^0(A)/CU_{\infty}^0(A)
                  \end{equation}
                  is a topological group isomorphism by means of (\ref{eq:Thomsen-isomorphism}), it follows that the map
                    \begin{equation}
                      U(A)/CU^0(A) \to \hku(A)
                    \end{equation}
                    is open as it will send open small neighbourhoods of the identity to open neighbourhoods (as sufficiently small neighbourhoods will be connected to the identity). 

                    Again, as the map $U(A)/U^0(A) \to K_1(A)$ is injective, we also have that $DU(A) \subseteq U^0(A)$.
                    Thus $DU^0(A) = DU(A)$ and $CU^0(A) = CU(A)$ in this case.
          \end{remark}

    Finally we finish by stating that unital C*-algebras satisfying
      \begin{enumerate}
        \item $\pi_1(U^0(A)) \to K_0(A)$ is surjective and
        \item $U(A)/U^0(A) \to K_1(A)$ is an isomorphism
      \end{enumerate}
      are very common. Indeed, this includes the class of stable rank one C*-algebras \cite[Theorem 3.3]{Rieffel87}, $\cZ$-stable C*-algebras \cite[Theorem 3]{Jiang97}, and tensor products with coronas over $\sigma$-unital C*-algebras \cite[Theorem 4.9]{Thomsen91}.

%\bibliography{bib}{}
%TODO: consistency in bibliography (full names or initials?). Also proofread bibliography: there are some errors and omissions.
\bibliographystyle{amsalpha}
\bibliography{biblio}

\end{document}